\documentclass[10pt,dvipsnames]{amsart}

\usepackage[utf8]{inputenc}
\usepackage{lmodern}
\usepackage[T1]{fontenc}
\usepackage[english]{babel}

\usepackage{latexsym}
\usepackage{amsmath,amssymb,amsthm}
\usepackage{mathrsfs}
\usepackage{stmaryrd}

\usepackage{enumerate}

\usepackage{rotating}

\usepackage{todonotes}
\usepackage{tikz-cd}
\usetikzlibrary{shapes}
\usepackage{tikz}

\usepackage{xcolor}
\usepackage{color}
\usepackage{hyperref}
\hypersetup{
	colorlinks=true,
	linkcolor=blue,
	citecolor=blue,
	urlcolor=blue
}

\usepackage{ulem}
\usepackage{soul}

\newtheorem{theoremAlph}{Theorem}

\newtheorem{theorem}{Theorem}[section]
\newtheorem{lemma}[theorem]{Lemma}	
\newtheorem{proposition}[theorem]{Proposition}
\newtheorem{corollary}[theorem]{Corollary}
\theoremstyle{definition}
\newtheorem{definition}[theorem]{Definition} 
\newtheorem{remark}[theorem]{Remark}

\theoremstyle{definition} 
\newtheorem*{ack}{Acknowledgements}

\numberwithin{equation}{section}

\newcommand{\C}{\mathbb{C}} 
\newcommand{\R}{\mathbb{R}} 
\newcommand{\Q}{\mathbb{Q}} 
\newcommand{\Z}{\mathbb{Z}} 
\newcommand{\N}{\mathbb{N}} 
\newcommand{\Quat}{\mathbb{H}}

\newcommand{\II}{\mathrm{I\!I}}
\newcommand{\Ric}{\textup{Ric}}

\newcommand{\ttimes}{\mathrel{\widetilde{\times} }}

\DeclareMathOperator{\im}{im}
\DeclareMathOperator{\intr}{Int}
\newcommand{\bq}{/ \hspace{-.12cm} /}

\makeatletter
\DeclareRobustCommand*\uell{\mathpalette\@uell\relax}
\newcommand*\@uell[2]{
	\setbox0=\hbox{$#1\ell$}
	\setbox1=\hbox{\rotatebox{10}{$#1\ell$}}
	\dimen0=\wd0 \advance\dimen0 by -\wd1 \divide\dimen0 by 2
	\mathord{\lower 0.1ex \hbox{\kern\dimen0\unhbox1\kern\dimen0}}
}

\begin{document}
	\title[\rmfamily Generalized Surgery and Positive Ricci Curvature]{\rmfamily Generalized Surgery on Riemannian Manifolds of Positive Ricci Curvature}
	\date{\today}
	\author{Philipp Reiser}
	\address{Institut f\"ur Algebra und Geometrie\\Karlsruher Institut f\"ur Technologie (KIT)\\Germany.}
	\curraddr{Department of Mathematics, University of Fribourg, Switzerland}
	\email{\href{mailto:philipp.reiser@unifr.ch}{philipp.reiser@unifr.ch}}
	\thanks{The author acknowledges funding by the Deutsche Forschungsgemeinschaft (DFG, German Research Foundation) -- 281869850 (RTG 2229) and grant GA 2050 2-1 within the SPP 2026 ``Geometry at Infinity''.}
	
	\normalem
	
	\begin{abstract}
		The surgery theorem of Wraith states that positive Ricci curvature is preserved under surgery if certain metric and dimensional conditions are satisfied. We generalize this theorem as follows: instead of attaching a product of a sphere and a disc, we glue a sphere bundle over a manifold with a so-called \emph{core metric}, a type of metric which was recently introduced by Burdick to construct metrics of positive Ricci curvature on connected sums. As applications we extend a result of Burdick on the existence of core metrics on certain sphere bundles and obtain new examples of 6-manifolds with metrics of positive Ricci curvature.
	\end{abstract}

	\maketitle

	\section{Introduction and Main Results}
	
	By the well-known surgery theorem of Gromov--Lawson \cite{GL80a} and Schoen--Yau \cite{SY79}, positive scalar curvature is preserved under surgeries of codimension at least 3. This result turned out to be a powerful tool to construct metrics of positive scalar curvature and, together with index theory of Dirac operators, led to a complete classification of closed, simply-connected manifolds with a metric of positive scalar curvature by Stolz \cite{St92}.
	
	So far no obstruction is known for a closed simply-connected manifold with a metric of positive scalar curvature to admit a metric of positive Ricci curvature. By \cite{GL80a} every closed simply-connected manifold of dimension 5, 6 or 7 admits a metric of positive scalar curvature, so potentially the same result for positive Ricci curvature could hold. However, there are relatively few known examples of manifolds that admit a complete Riemannian metric with this curvature condition. Existence results were obtained by using bundles and warping techniques (see \cite{BB78}, \cite{GPT98}, \cite{Na79}, \cite{Po75}), and by Lie group actions, in particular actions of low cohomogeneity (see \cite{GZ02}, \cite{Na79}), and biquotients (see \cite{ST04}). In this context we also refer to \cite{SW15}. Further results were obtained by making use of additional structures such as Kähler geometry (see \cite{Ya78}) and Sasakian geometry (see \cite{BG02, BG06a}, \cite{BG06}, \cite{BGN03}). Concerning surgery, it is open whether a similar surgery theorem as for positive scalar curvature holds. By the classical theorem of Bonnet--Myers the connected sum of two closed non-simply-connected $n$-dimensional manifolds (which is the result of surgery in codimension $n$), cannot admit a metric of positive Ricci curvature as the fundamental group of this manifold is infinite. For lower codimension, however, it is still possible that a surgery theorem in the same generality as in the case of positive scalar curvature holds.
	
	Under additional assumptions, Sha and Yang \cite{SY91} proved a surgery theorem for positive Ricci curvature, which was later extended by Wraith \cite{Wr98}. For that, following \cite{Wr98}, we suppose we have the following:
	\begin{enumerate}
		\item[$(S1)$] A Riemannian manifold $(M^{p+q-1},g_M)$ of positive Ricci curvature.
		\item[$(S2)$] An isometric embedding $\iota\colon S^{p-1}(\rho)\times D_R^{q}(N)\hookrightarrow \intr(M)$, where $S^{p-1}(\rho)$ denotes the round $(p-1)$-sphere of radius $\rho>0$ and $D_R^{q}(N)$ denotes the ball of radius $R>0$ in $S^{q}(N)$.
		\item[$(S3)$] A smooth map $T\colon S^{p-1}\to SO(q)$, which induces a diffeomorphism\linebreak $\tilde{T}\colon S^{p-1}\times S^{q-1}\to S^{p-1}\times S^{q-1}$ defined by $(x,y)\mapsto (x,T_x(y))$.
	\end{enumerate}
	
	\begin{theorem}[{\cite[Theorem 0.3]{Wr98}}]
		\label{T:WR-SURGERY}
		Under assumptions (S1)-(S3) let $p\geq q\geq3$. Then there is a constant $\kappa=\kappa(p,q,R/N,T)>0$ such that if $\frac{\rho}{N}<\kappa$, then the manifold
		\[\hat{M}=M\setminus\im(\iota)^\circ\cup_{\tilde{T}} (D^p\times S^{q-1}) \]
		admits a metric of positive Ricci curvature.
	\end{theorem}
	The metric constructed on $\hat{M}$ coincides outside a neighborhood of the gluing area with the restriction of $g_M$ to $M\setminus\im(\iota)$ and with $D^{p}_{R_1}(N_1)\times S^{q-1}(\rho_1)$ near the center of $D^p\times S^{q-1}$ for some $\rho_1,N_1,R_1>0$. The quotient $\frac{\rho_1}{N_1}$ can be bounded from above by any constant $\kappa^\prime>0$, which then gives an additional dependency for $\kappa$, while the quotient $\frac{R_1}{N_1}$ is constant independently of $M,\iota,T$ and $\kappa^\prime$, see \cite[Proposition 0.4]{Wr98}. This form of the metric allows us to apply Theorem \ref{T:WR-SURGERY} again to the manifold $\hat{M}$, provided $p=q$, which leads to the following theorem by Wraith.
	
	\begin{theorem}[{\cite[Theorems 2.2 and 2.3]{Wr97}}]
		\label{T:WR-PLUMBING}
		Let $W$ be the manifold obtained by plumbing linear $n$-disc bundles over $n$-spheres according to a simply-connected graph or by plumbing together two disc bundles over spheres (where fiber and base dimension may differ). If the fiber and base dimensions are at least 3 then $\partial W$ admits a metric of positive Ricci curvature.
	\end{theorem}
	Recall that plumbing is a procedure that glues disc bundles to each other in a certain way, see Section \ref{S:PLUMBING}. The class of manifolds obtained as boundaries of plumbings as in Theorem \ref{T:WR-PLUMBING} is large and contains many interesting examples, including all homotopy spheres that bound parallelizable manifolds \cite[Theorem 2.1]{Wr97} and many highly-connected manifolds \cite{CW17}.
	
	We generalize Theorems \ref{T:WR-SURGERY} and \ref{T:WR-PLUMBING} by replacing the base spheres with manifolds admitting so-called \emph{core metrics}, a notion which was introduced by Burdick \cite{Bu19}.
	\begin{definition}
		Let $M^n$ be a smooth manifold. A Riemannian metric $g$ on $M$ is called a \emph{core metric} if it has positive Ricci curvature and if there is an embedding $\varphi\colon D^n\hookrightarrow \intr(M)$ such that $g|_{\varphi(S^{n-1})}$ is the round metric of radius 1 and such that the second fundamental form $\II_{\varphi(S^{n-1})}$ is positive definite with respect to the inward pointing normal vector of $S^{n-1}\subseteq D^n$.
	\end{definition}
	Since there exist different sign conventions for the second fundamental form, we give a simple example: An appropriate scalar multiple of the round metric on $S^n$ is a core metric, where the embedded disc can be any geodesic ball whose interior contains a hemisphere.
	
	The notion of core metrics is based on work by Perelman \cite{Pe97} and its relevance is illustrated by the theorem below.
	\begin{theorem}[{\cite[Theorem B]{Bu19}}]
		Let $M_i^n$, $1\leq i\leq k$ be manifolds that admit core metrics. If $n\geq 4$, then $\#_i M_i$ admits a metric of positive Ricci curvature.
	\end{theorem}
	The main goal of \cite{Bu19a}, \cite{Bu19} and \cite{Bu20} was to construct manifolds which admit core metrics. The main results are given as follows.
	\begin{theorem}[\cite{Pe97}, \cite{Bu19a}, \cite{Bu19}, \cite{Bu20}, \cite{Bu20a}]
		\label{T:CORE_METRICS}
		The following manifolds admit core metrics:
		\begin{enumerate}
			\item $S^n$, if $n\geq2$;
			\item $\C P^n$, $\Quat P^n$ and $\mathbb{O}P^2$;
			\item $M^n_1\# M^n_2$ if $n\geq4$ and $M_1$, $M_2$ admit core metrics;
			\item Total spaces of linear sphere bundles $E\to B$ with fiber and base dimension at least $3$ if $B$ admits a core metric;\footnote{In \cite{Bu20} this result is stated with no restriction on the base dimension. However, the proof given in \cite{Bu20} does not work if the base is 2-dimensional, since the metric on $M_1$ in \cite[Definition 3]{Bu20} does not have positive Ricci curvature if $n=1$. In Theorem \ref{T:2-SPHERE_BUNDLES}, by using Theorem \ref{T:PLUMBING}, we fix this by giving an alternative proof provided the fiber dimension is at least 4. The author would like to thank Bradley Burdick for helpful discussions.}
			\item $\partial W$ for $W$ as in Theorem \ref{T:WR-PLUMBING} if one of the disc bundles has fiber dimension at least 4.
		\end{enumerate}
	\end{theorem}

	Our main result is the following.
	\begin{theoremAlph}
		\label{T:GEN_SURG}
		Under the assumptions (S1) and (S2) let $B^p$ be a manifold with a core metric $g_B$, let $E\xrightarrow{\pi} B$ be a linear $S^{q-1}$-bundle, and let $r>0$ be sufficiently small. If $p,q\geq3$ then there is a constant $\kappa=\kappa(p,q,R/N,g_B,r)>0$ such that if $\frac{\rho}{N}<\kappa$, then the manifold
		\[\hat{M}=M\setminus\im(\iota)^\circ\cup_\partial \pi^{-1}(B\setminus \varphi(D^p)^\circ) \]
		admits a metric of positive Ricci curvature. This metric coincides outside a neighborhood of the gluing area with a submersion metric on $E$ with totally geodesic and round fibers of radius $r$ and with a scalar multiple of the metric $g_M$ on $M$.
	\end{theoremAlph}
	More precisely, the dependence of $\kappa$ on the metric $g_B$ is completely determined by the smallest eigenvalue of the second fundamental form $\II_{\varphi(S^{p-1})}$.
	
	Theorem~\ref{T:GEN_SURG} differs from Theorem~\ref{T:WR-SURGERY} in two aspects: First, the requirement on the dimensions $p$ and $q$ is more general and second, the surgery construction itself is generalized. Note that under the assumptions $B=S^p$ and $p\geq q$ we precisely obtain Theorem \ref{T:WR-SURGERY} from Theorem \ref{T:GEN_SURG}.
	
	Theorem \ref{T:GEN_SURG} can now be used to prove the following generalization of Theorem \ref{T:WR-PLUMBING} and of item~(5) of Theorem \ref{T:CORE_METRICS}. Note that, since we merely need to require $p,q\geq3$ in Theorem~\ref{T:GEN_SURG}, we do not need the assumption from Theorem~\ref{T:WR-PLUMBING} that base and fiber dimensions coincide.
	\begin{theoremAlph}
		\label{T:PLUMBING}
		Let $W$ be the manifold obtained by plumbing linear disc bundles $\overline{E}_i$ with compact base manifolds $B_i$, $1\leq i\leq k$, according to a simply-connected graph. Then the following holds:
		\begin{enumerate}
			\item If $B_1$ admits a metric of positive Ricci curvature and $B_i$ for $i\geq 2$ admits a core metric, then $\partial W$ admits a metric of positive Ricci curvature provided $\dim(B_i)\geq3$ for all $i$ (or, equivalently, all fiber dimensions are at least 3).
			\item If $B_1$ also admits a core metric with $\dim(B_1)\geq3$ and the fiber dimension of $\overline{E}_1$ is at least 4, then $\partial W$ admits a core metric.
		\end{enumerate} 
	\end{theoremAlph}

	As an application we will show that Theorem \ref{T:WR-PLUMBING} can be used to extend item (4) of Theorem \ref{T:CORE_METRICS} as follows.
	\begin{theoremAlph}
		\label{T:2-SPHERE_BUNDLES}
		Let $E\to B^q$ be a linear $S^p$-bundle and suppose that
		\begin{itemize}
			\item $p=2$ and $q\geq4$, or
			\item $q=2$ and $p\geq4$.
		\end{itemize}
		 If $B$ is closed and admits a core metric, then $E$ admits a core metric.
	\end{theoremAlph}

	The proof of Theorem~\ref{T:2-SPHERE_BUNDLES} relies on alternative descriptions of the total spaces of these bundles in terms of plumbings that satisfy the assumptions of Theorem \ref{T:PLUMBING}.

	Finally, we consider applications in dimension 6. An immediate consequence of Theorem \ref{T:CORE_METRICS} and Theorem \ref{T:2-SPHERE_BUNDLES} is the corollary below.
	\begin{corollary}
		\label{C:CORE_METRICS_DIM6}
		The following 6-manifolds admit core metrics:
		\begin{enumerate}
			\item $S^6$;
			\item $S^2\times S^4$ and $S^2\ttimes S^4$, the unique non-trivial $S^4$-bundle over $S^2$;
			\item $S^3\times S^3$;
			\item Linear $S^2$-bundles over $\#_{i=1}^k (\pm\C P^2)$ for all $k\in\N_0$ (for $k=0$ we obtain $S^4$);
			\item Finitely many connected sums of the manifolds in (1)-(4).
		\end{enumerate}
	\end{corollary}
	Note that $\C P^3$ has the structure of a linear $S^2$-bundle over $S^4$, so it is contained in item (4). The manifolds in this item will be analyzed in Section \ref{SS:S2-BUNDLES} below.
	
	Corollary \ref{C:CORE_METRICS_DIM6} provides infinitely many new examples, both spin and non-spin, of 6-manifolds with a metric of positive Ricci curvature, see Proposition \ref{P:NEW} and Remark \ref{R:NEW} below.
	
	In Corollary \ref{C:CORE_METRICS_DIM6} we did not consider applications of Theorem \ref{T:PLUMBING}. In fact, by using plumbings as in Theorem \ref{T:PLUMBING}, we can construct a much larger class of 6-manifolds with a metric of positive Ricci curvature than the class of manifolds in Corollary \ref{C:CORE_METRICS_DIM6}. This will be carried out in a future paper \cite{Re22a}.

	This paper is organized as follows. In Section \ref{S:PREL} we recall the plumbing construction and the Vilms construction to obtain submersion metrics on fiber bundles. We proceed in Section \ref{S:THM_A} with the proof of Theorem \ref{T:GEN_SURG} and in Section \ref{S:THMS_BC} with the proof of Theorems \ref{T:PLUMBING} and \ref{T:2-SPHERE_BUNDLES}. Finally, in Section \ref{S:6-MANIFOLDS} we consider applications in dimension~6.
	
	\begin{ack}
		The author would like to thank Fernando Galaz-Garc\'ia and Wilderich Tuschmann for many helpful conversations and comments. The author would also like to thank the Department of Mathematical Sciences of Durham University for its hospitality during a visit where parts of the work contained in this paper were carried out. Finally, the author would like to thank the anonymous referees for their suggestions, which significantly improved the presentation of this article.
	\end{ack}
	
	\section{Preliminaries}
	\label{S:PREL}
	In the following the term \emph{manifold} will denote a smooth manifold, possibly with boundary. Unless stated otherwise all maps between manifolds and all fiber bundles over manifolds are assumed to be smooth. If a manifold $M$ is assumed to be oriented, then $-M$ denotes the same manifold with the reverse orientation. If not stated otherwise, we use (co-)homology with coefficients in $\Z$.
	
	\subsection{Plumbing}
	\label{S:PLUMBING}
	In this section we briefly recall the plumbing construction. It was introduced by Milnor \cite{Mi58} to construct manifolds with prescribed intersection form. We also refer to \cite{HM68}, \cite[Section 5]{Br72} and \cite[Section 2]{CW17} for further details on plumbing.
	
	Let $D^q\hookrightarrow \overline{E}_1\xrightarrow{\pi_1} B_1^p$ and $D^p\hookrightarrow \overline{E}_2\xrightarrow{\pi_2} B_2^q$ be oriented linear disc bundles over oriented and connected manifolds $B_i$ such that fibers, base and total space are oriented compatibly. Let $D_1^p\hookrightarrow \intr(B_1)$, $D_2^q\hookrightarrow \intr(B_2)$ be orientation preserving embeddings. Since discs are contractible, we can identify $\pi_1^{-1}(D_1^p)\cong D_1^p\times D^q$ and $\pi_2^{-1}(D_2^q)\cong D_2^q\times D^p$. Now choose diffeomorphisms $\phi_1\colon D_1^p\to D^p$ and $\phi_2\colon D^q\to D_2^q$ that both preserve or both reverse the orientation and let $\overline{E}_1\square \overline{E}_2$ be the space obtained from $\overline{E}_1\sqcup \overline{E}_2$ by identifying $\pi_1^{-1}(D_1^p)$ and $\pi_2^{-1}(D_2^q)$ via the diffeomorphism
	\begin{align*}
		I\colon D_1^p\times D^q&\to D_2^q\times D^p\\
		(x,y)&\mapsto(\phi_2(y),\phi_1(x)).
	\end{align*}
	The space $\overline{E}_1\square \overline{E}_2$ has a manifold structure after smoothing out the corners which arise at the boundary of the identification area. It can now be shown that the diffeomorphism type of $\overline{E}_1\square \overline{E}_2$ only depends on whether the diffeomorphisms $\phi_i$ preserve or reverse the orientation. We say that $\overline{E}_1\square \overline{E}_2$ is the manifold obtained by plumbing $\overline{E}_1$ and $\overline{E}_2$ with sign $+1$ (sign $-1$) if both maps $\phi_i$ are orientation preserving (orientation reversing). Independent of the sign, the manifold $\overline{E}_1\square \overline{E}_2$ is oriented compatibly with $\overline{E}_1$ and $\overline{E}_2$ if $p$ or $q$ is even, and $\overline{E}_1\square \overline{E}_2$ is oriented compatibly with $\overline{E}_1$ and $-\overline{E}_2$ if $p$ and $q$ are odd.
	
	We can repeat the process of plumbing by choosing multiple embedded discs disjoint from each other. A manifold obtained by plumbing multiple disc bundles can then be characterized by a labeled graph, where we label each edge which either $+1$ or $-1$. Each vertex corresponds to a disc bundle and each edge corresponds to plumbing according to the sign the edge is labeled with. By collapsing the fibers we see that the manifold obtained in this way deformation retracts onto a sequence of one-point unions of the base manifolds, where we connect two base manifolds if the corresponding vertices are connected by an edge.
	
	We will be interested in the boundaries of plumbings. Let $E_i=\partial\overline{E}_i$ be the sphere bundle of the disc bundle. It is easily verified that
	\begin{equation}
		\label{EQ:PLUMB_BOUND}
		\partial(\overline{E}_1\square \overline{E}_2)\cong(E_1\setminus \pi_1^{-1}(D_1^p)^\circ)\cup_{I|_{S^{p-1}\times S^{q-1}}} (E_2\setminus\pi_2^{-1}(D_2^q)^\circ).
	\end{equation}
	In particular, if one of the $\overline{E}_i$, say $\overline{E}_2$, is the trivial bundle $S^q\times D^p$, then $\partial(\overline{E}_1\square \overline{E}_2)$ is obtained by surgery along a fiber sphere of $E_1$.
	
	The decomposition \eqref{EQ:PLUMB_BOUND} shows that the boundary of a plumbed manifold is the same manifold we obtain in Theorem \ref{T:GEN_SURG} by setting $M=E_1$ and $E=E_2$. Hence, by repeatedly applying Theorem \ref{T:GEN_SURG}, we can equip the boundary of a manifold obtained by multiple plumbings with a metric of positive Ricci curvature if we can make sure that the assumptions of Theorem \ref{T:GEN_SURG} are satisfied in each step. We will show that this is always possible if we plumb according to a simply-connected graph, see Section~\ref{S:THMS_BC}.
	
	\subsection{Metrics of Positive Ricci Curvature on Sphere Bundles}
	To construct manifolds of positive Ricci curvature on plumbed manifolds we will use the following special case of the so-called \emph{Vilms construction}:
	\begin{proposition}[{\cite[Theorem 9.59]{Be87}}]
		\label{P:VILMS}
		Let $E\xrightarrow{\pi} B$ be a linear sphere bundle over a Riemannian manifold $(B,g_B)$. Let $\theta$ be a connection on the principal $\mathrm{O}(p)$-bundle corresponding to $E$ and let $r>0$. Then there is precisely one metric $g_E(r,\theta)$ on $E$ such that the map $\pi$ is a Riemannian submersion with totally geodesic and round fibers of radius $r>0$ and horizontal distribution associated with $\theta$.
	\end{proposition}
	The Ricci curvatures of the metric $g_E(r,\theta)$ are given in \cite[Theorem 9.70]{Be87}. From the formulas given there one sees that the Ricci curvatures are close to those of $B$ and the fiber $S^{p-1}$ if the radius $r$ of the fibers is small. Hence, we obtain the following.
	\begin{proposition}[{\cite[Theorem 9.70]{Be87}}]
		\label{P:VILMS_RICCI}
		Let $E\xrightarrow{\pi} B$ be a linear sphere bundle over a Riemannian manifold $(B,g_B)$ with a connection $\theta$ on the corresponding principal $\mathrm{O}(p)$-bundle. If $B$ is compact and the metric $g_B$ has positive Ricci curvature, then there is a constant $r_0>0$ such that the metric $g_E(r,\theta)$ has positive Ricci curvature for all $r\in(0,r_0)$.
	\end{proposition}
	
	\section{Proof of Theorem \ref{T:GEN_SURG}}
	\label{S:THM_A}
	
	In this section we prove Theorem \ref{T:GEN_SURG}. First we decompose the manifold $\hat{M}$ as follows:
	\begin{equation}
		\label{EQ:GLUE_DEC}
		\hat{M}\cong \pi^{-1}(B\setminus \varphi(D^p)^\circ)\cup_{S^{p-1}\times S^{q-1}}(I\times S^{p-1}\times S^{q-1})\cup_{S^{p-1}\times S^{q-1}}(M\setminus\mathrm{im}(\iota)^\circ),
	\end{equation} 
	where $I$ is a closed interval. The strategy is to define a suitable metric of positive Ricci curvature on each part and then glue them using the following theorem due to Perelman \cite{Pe97}.
	\begin{theorem}[{\cite[Section 4]{Pe97}, \cite[Theorem 2]{BWW19}}]
		\label{T:PER_GLUING}
		Let $M_1, M_2$ be Riemannian manifolds of positive Ricci curvature. Suppose that there is an isometry $\phi\colon \partial_c M_1\to \partial_c M_2$ between two boundary components $\partial_c M_1\subseteq \partial M_1$ and $\partial_c M_2\subseteq\partial M_2$ and that the sum of the second fundamental forms $\II_{\partial_c M_1}+\phi^*\II_{\partial_c M_2}$ is positive semi-definite. Then there is a metric of positive Ricci curvature on the manifold $M_1\cup_\phi M_2$ that agrees with the original metrics on $M_1$ and $M_2$ outside an arbitrarily small neighborhood of the gluing area.
	\end{theorem}
	The theorem is originally stated with the assumption that the sum of second fundamental forms is positive definite. However, if this sum is merely positive semi-definite, we can perturb the metrics near the boundary slightly to increase the second fundamental form while keeping the curvature bounds and the induced metric on the boundary, see e.g.\ \cite[Proof of Proposition 1.2.11]{Bu19a}.

	If the metrics in Theorem \ref{T:PER_GLUING} are warped product metrics, then, by observing that the metric we obtain after gluing is again a warped product metric, we obtain the following special case.
	\begin{corollary}
	\label{C:WARP_GLUING}
		Let $J$ be an interval and let $(M_1,g_1),\dots,(M_k,g_k)$ be Riemannian manifolds. Let $f_1,\dots,f_k\colon J\to \R_{>0}$ be continuous functions which are smooth on $J\setminus\{x_1,\dots, x_l\}$, where $x_1,\dots,x_l\in J$ are interior points. If the metric
		\[g=dt^2+f_1(t) g_1 +\dots+f_k(t) g_k\]
		on $J\times M_1\times\dots\times M_k$ has positive Ricci curvature for all $t\in J\setminus\{x_1,\dots, x_l\}$ and if
		\[f_{i-}^\prime(x_j)\geq f_{i+}^\prime(x_j) \]
		for all $i,j$, then we can smooth the functions $f_1,\dots,f_k$ on an arbitrarily small neighborhood of each $x_j$ such that the resulting metric has positive Ricci curvature.
	\end{corollary}
	We note that the case $k=1$ in Corollary \ref{C:WARP_GLUING} had already been considered in \cite[Observation 1.5]{Wr98}.
	
	\subsection{The metrics on $\pi^{-1}(B\setminus \varphi(D^p)^\circ)$ and $M\setminus\mathrm{im}(\iota)^\circ$.}
	To obtain a metric on $E$ we first choose a connection on the principal $\mathrm{O}(q)$-bundle corresponding to $E$ which is flat on the embedded disc $\varphi(D^p)\subseteq B$. The metric $g_E(r,\theta)$ from Proposition \ref{P:VILMS}, which has positive Ricci curvature for $r>0$ sufficiently small by Proposition \ref{P:VILMS_RICCI}, is then, when restricted to $\varphi(D^p)$, a product of the form
	\[g_E(r,\theta)|_{\pi^{-1}(\varphi(D^p))}=g_B|_{\varphi(D^{p})}+r^2\cdot ds_{q-1}^2 \]
	and over the boundary $\varphi(S^{p-1})$ we have
	\[g_E|_{\pi^{-1}(\varphi(S^{p-1}))}=ds^2_{p-1}+r^2\cdot ds_{q-1}^2. \]
	Since the metric is a product over $\varphi(D^p)$, the second fundamental form on\linebreak $\pi^{-1}(\varphi(S^{p-1}))\cong S^{p-1}\times S^{q-1}$ with respect to this product structure is given by
	\begin{equation}
		\label{EQ:II_PHI}
		\II_{\pi^{-1}(\varphi(S^{p-1}))}=
		\begin{pmatrix}
			\mathrm{I\!I}_{\varphi(S^{p-1})}&0\\
			0&0
		\end{pmatrix}.
	\end{equation}
	Since $g_B$ is a core metric, we have $\mathrm{I\!I}_{\varphi(S^{p-1})}>0$. Hence, the smallest eigenvalue of $\II_{\varphi(S^{p-1})}$, which we denote by $\lambda$, is positive. Note that scaling the metric $g_E(r,\theta)$ by a factor $\alpha>0$ has the effect that $\lambda$ gets multiplied by $\alpha^{-1}$.
	
	By assumption, the metric on $\iota(S^{p-1}\times S^{q-1})\cong S^{p-1}\times S^{q-1}$ is the product metric
	\[g_M|_{\iota(S^{p-1}\times S^{q-1})}=\rho^2\cdot ds_{p-1}^2+N^2\sin^2\left(\frac{R}{N}\right)ds_{q-1}^2 \]
	and the second fundamental form with respect to this product structure is given by
	\begin{equation}
		\label{EQ:II_IOTA}
		\II_{\iota(S^{p-1}\times S^{q-1})}=
		\begin{pmatrix}
			0&0\\
			0&-\frac{1}{N}\cot\left(\frac{R}{N}\right)
		\end{pmatrix}.
	\end{equation}
	In general, the value $\frac{R}{N}$ can be very small, in which case $\II_{\iota(S^{p-1}\times S^{q-1})}$ is negative semi-definite. 

	\subsection{The structure of the metric on $I\times S^{p-1}\times S^{q-1}$.}
	We will equip the middle part of \eqref{EQ:GLUE_DEC} with a metric of positive Ricci curvature such that we can glue it to the other parts using Theorem \ref{T:PER_GLUING}. The metric will be a doubly warped product metric, i.e.\ it will be given by
	\[g_{f,h}=dt^2+h^2(t)ds_{p-1}^2+f^2(t)ds_{q-1}^2, \]
	where $f,h\colon\R_{\geq0}\to\R_{>0}$ are smooth functions. The second fundamental form of $g_{f,h}$ at a slice $t\geq0$ with respect to $\partial_t$ is given by
	\begin{equation}
		\label{EQ:II_WARP}
		\II_{\{t\}\times (S^{p-1}\times S^{q-1})}=
		\begin{pmatrix}
			\frac{h^\prime(t)}{h(t)} & 0\\
			0 & \frac{f^\prime(t)}{f(t)}
		\end{pmatrix}.
	\end{equation}
	Now, in order to glue according to the decomposition \eqref{EQ:GLUE_DEC} using Theorem \ref{T:PER_GLUING}, we impose the following boundary conditions:
	\begin{align}
		h(0)&=\alpha & f(0)&=\alpha r\label{EQ:BOUND_COND_0}\\
		h^\prime(0)&\leq\lambda & f^\prime(0) & \leq0\label{EQ:BOUND_COND_0'}\\
		h(t_0)&=\beta\rho & f(t_0)&=\beta N\sin\left(R/N\right)\label{EQ:BOUND_COND_t0}\\
		h^\prime(t_0)&\geq0 & f^\prime(t_0)&\geq\cos\left(R/N\right)\label{EQ:BOUND_COND_t0'}.
	\end{align}
	Here $\alpha,\beta, t_0>0$ can be chosen arbitrarily. Furthermore, in order to use Theorem \ref{T:PER_GLUING}, the metric $g_{f,h}$ needs to have positive Ricci curvature. For that, let $V\in TS^{p-1}$, $W\in TS^{q-1}$ be unit length vectors (with respect to the metric $g_{f,h}$). Then the Ricci curvatures are given as follows, see e.g.\ \cite[Proposition 4.2]{Wr07}:
	\begin{align}
		\Ric(\partial_t,\partial_t)&=-(p-1)\frac{h^{\prime\prime}}{h}-(q-1)\frac{f^{\prime\prime}}{f}\label{EQ:D_WARP_T},\\
		\Ric(V,V)&=-\frac{h^{\prime\prime}}{h}+(p-2)\frac{1-(h^\prime)^2}{h^2}-(q-1)\frac{h^\prime f^\prime}{hf}\label{EQ:D_WARP_V},\\
		\Ric(W,W)&=-\frac{f^{\prime\prime}}{f}+(q-2)\frac{1-(f^\prime)^2}{f^2}-(p-1)\frac{h^\prime f^\prime}{h f}\label{EQ:D_WARP_W},\\
		\Ric(\partial_t,V)&=\Ric(\partial_t,W)=\Ric(V,W)=0\nonumber.
	\end{align}
	Figure 1 contains a sketch of how the graph of such functions $h$ and $f$ would typically look like.
	\begin{lemma}
		\label{L:THM_A}
		If the functions $h$ and $f$ satisfy \eqref{EQ:BOUND_COND_0}--\eqref{EQ:BOUND_COND_t0'} and the Ricci curvatures \eqref{EQ:D_WARP_T}--\eqref{EQ:D_WARP_W} are positive, then the manifold $\hat{M}$ has a metric of positive Ricci curvature as claimed in Theorem \ref{T:GEN_SURG}.
	\end{lemma}
	\begin{proof}
		We scale the metric $g_E$ by $\alpha$, so $(\pi^{-1}(B\setminus \varphi(D^p)^\circ),\alpha^2g_E(r,\theta))$ and $(I\times S^{p-1}\times S^{q-1},g_{f,h})$ have an isometric boundary component by \eqref{EQ:BOUND_COND_0}. Scaling by $\alpha$ has the effect that the second fundamental form on $\varphi(S^{p-1})$ becomes bounded from below by $\frac{\lambda}{\alpha}$. Hence, by \eqref{EQ:BOUND_COND_0'}, \eqref{EQ:II_PHI} and \eqref{EQ:II_WARP} the requirements of Theorem \ref{T:PER_GLUING} are satisfied for this boundary component (note that we need to reverse the signs in \eqref{EQ:II_WARP} since $\partial_t$ is the inward normal vector on this boundary component). For the other boundary component we proceed similarly, i.e.\ we rescale the metric $g_M$ by $\beta$ so we have isometric boundary components by \eqref{EQ:BOUND_COND_t0}. Then by \eqref{EQ:BOUND_COND_t0'}, \eqref{EQ:II_IOTA} and \eqref{EQ:II_WARP} the requirements of Theorem \ref{T:PER_GLUING} are satisfied. Now we apply Theorem \ref{T:PER_GLUING} to glue according to the decomposition \eqref{EQ:GLUE_DEC} and rescale the resulting metric by $\frac{1}{\alpha}$.
	\end{proof}

		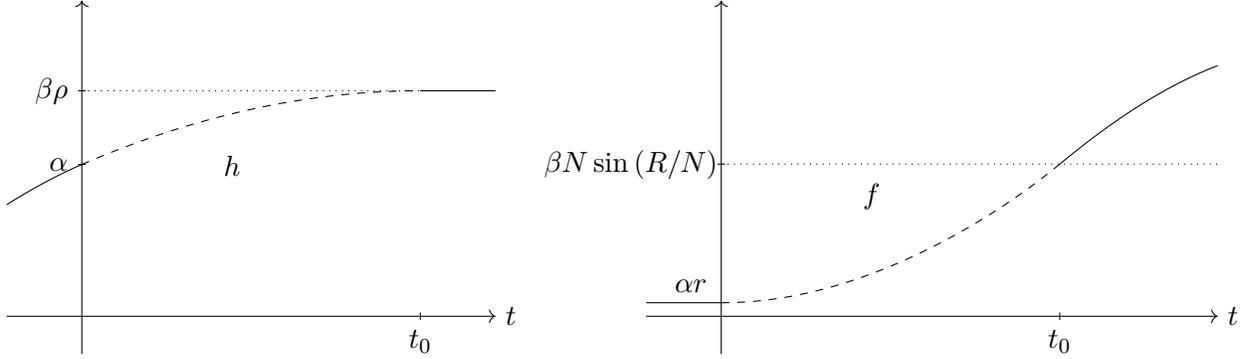
\begin{figure}
			\centering
			\begin{tikzpicture}[scale=0.775, transform shape]
				\draw[->] (-1, 0) -- (5.5,0) node[right] {$t$};
				\draw[->] (0, -0.5) -- (0,4.2) node[above] {};
				\draw (4.5,0.05) -- (4.5,-0.05) node[below] {$t_0$};
				\draw (-0.05,2.01953) -- (0.05,2.01953) node[left] {};
				\draw (-0.05,2.01953) node[left] {$\alpha$};
				\draw (-0.05,3) -- (0.05,3) node[left] {};
				\draw (-0.05,3) node[left] {$\beta\rho$};

				\node at (2,2) {$h$};
				\draw[scale=3, domain=-0.333:0, smooth, variable=\x, black] plot ({\x}, {sin(deg(\x+1))*0.8
				});
				\draw[scale=3, domain=0:1.5, smooth, variable=\x, black, dashed] plot ({\x}, {-0.00156553*\x*\x*\x-0.140558*\x*\x+cos(deg(1))*0.8*\x+sin(deg(1))*0.8
				});
				\draw[scale=3, domain=1.5:1.8333, smooth, variable=\x, black] plot ({\x}, {1
				});
				\draw[scale=3, domain=0:1.5, smooth, variable=\x, black, dotted] plot ({\x}, {1
				});
			\end{tikzpicture}
			\begin{tikzpicture}[scale=0.775, transform shape]
				\draw[->] (-1, 0) -- (6.6,0) node[right] {$t$};
				\draw[->] (0, -0.5) -- (0,4.2) node[above] {};
				\draw (4.5,0.05) -- (4.5,-0.05) node[below] {$t_0$};
				\draw (-0.05,0.18) node[above left] {$\alpha r$};
				\draw (-0.05,2.02466) -- (0.05,2.02466) node[left] {};
				\draw (0.09,2.02466) node[left] {$\beta N\sin\left(R/N\right)$};
				
				\node at (2,1.6) {$f$};
				\draw[scale=3, domain=0:1.5, smooth, variable=\x, black, dashed] plot ({\x}, {((7/720*\x*\x*\x*\x*\x*\x-1/24*\x*\x*\x*\x+1/2*\x*\x+1)-0.9)*0.6
				});
				\draw[scale=3, domain=1.5:2.2, smooth, variable=\x, black] plot ({\x}, {1.20448*sin(deg((\x-0.783619)/1.20448))
				});
				\draw[scale=3, domain=-0.333:0, smooth, variable=\x, black] plot ({\x}, {0.1*0.6
				});
				\draw[scale=3, domain=0:2.2, smooth, variable=\x, black, dotted] plot ({\x}, {0.674886
				});
			\end{tikzpicture}
			\caption{Sketch of the graph of the functions $h$ and $f$. The dashed lines correspond to the part on which we have to construct the functions $h$ and $f$ such that they satisfy the boundary conditions \eqref{EQ:BOUND_COND_0}--\eqref{EQ:BOUND_COND_t0'} and such that the Ricci curvatures \eqref{EQ:D_WARP_T}--\eqref{EQ:D_WARP_W} are positive.}
		\end{figure}
	
	\subsection{Construction of the functions $f$ and $h$.}
	
	In this subsection we construct functions $h$ and $f$ that satisfy the boundary conditions \eqref{EQ:BOUND_COND_0}--\eqref{EQ:BOUND_COND_t0'} and so that the Ricci curvatures \eqref{EQ:D_WARP_T}--\eqref{EQ:D_WARP_W} are positive. The construction follows a similar strategy as in \cite{SY91} and \cite{Wr98} by defining the functions as solutions of suitable differential equations. When compared to the construction in \cite{Wr98}, due to the more general setting in Theorem~\ref{T:GEN_SURG}, both the boundary conditions \eqref{EQ:BOUND_COND_0} and \eqref{EQ:BOUND_COND_0'} and the conditions on the numbers $p$ and $q$ are more general. As a consequence, we need to choose different differential equations to cover this more general situation.
	
	\begin{definition}
		Let $h_0\colon[0,\infty)\to\R$ be the unique smooth function satisfying
		\begin{align*}
			h_0^\prime&=e^{-\frac{1}{2}h_0^2},\\
			h_0(0)&=\sqrt{-2\ln\left(\min\left(\lambda,\frac{1}{2}\right)\right)}.
		\end{align*}
	\end{definition}
	We use the minimum of $\lambda$ and $\frac{1}{2}$ to cover the case $\lambda\geq1$, in which $-2\ln(\lambda)$ would be non-positive.
	\begin{lemma}
		We have
		\begin{enumerate}
			\item $h_0^\prime(0)=\min\left(\lambda,\frac{1}{2}\right)\leq\lambda$,
			\item $h_0,h_0^\prime>0$,
			\item $h_0^{\prime\prime}=-h_0e^{-h_0^2}<0$,
			\item $\lim_{t\to\infty}h_0(t)=\infty$.
		\end{enumerate}
	\end{lemma}
	\begin{proof}
		We show that $\lim_{t\to\infty}h_0(t)=\infty$, the remaining statements are easily verified. Since $h_0^\prime>0$, the function $h_0$ converges to a limit $L\in(0,\infty]$. Suppose $L<\infty$. Then $\lim_{t\to\infty}h_0^\prime(t)=0$. By the definition of $h_0$ we have $\lim_{t\to\infty} h_0^\prime(t)=e^{-\frac{1}{2}L^2}>0$, which is a contradiction. Hence $L=\infty$.
	\end{proof}
	\begin{definition}
		For $C\in(0,1)$ let $f_C\colon[0,\infty)\to\R$ be the unique smooth function satisfying
		\begin{align*}
			f_C^{\prime\prime}&=C e^{-h_0^2}f_C,\\
			f_C(0)&=1,\\
			f^\prime_C(0)&=0.
		\end{align*}
	\end{definition}
	\begin{lemma}
		\label{L:FC_PROPERTIES}
		We have
		\begin{enumerate}
			\item $f_C,f_C^\prime,f_C^{\prime\prime}>0$ on $(0,\infty)$,
			\item $\lim_{t\to\infty}f_C(t)=\lim_{t\to\infty}f_C^\prime(t)=\infty$,
			\item $\lim_{t\to\infty}f_C(t)h_0^\prime(t)=0$,
			\item $\frac{f^\prime_C}{f_C h_0 h_0^\prime}\in[0,1]$.
		\end{enumerate}
	\end{lemma}
	\begin{proof}
		(1) By definition we have $f_C^{\prime\prime}(0)>0$ and hence $f_C^\prime(t)>0$ for small $t$. Now suppose there is $t>0$ such that $f_C^{\prime\prime}(t)=0$. Let $t_0$ be the smallest such $t$, which is positive by the initial conditions. The equation $f_C^{\prime\prime}(t_0)=0$ implies $f_C(t_0)=0$, hence there is $t_1\in(0,t_0)$ such that $f_C^{\prime\prime}(t_1)=0$ since otherwise $f_C^{\prime\prime}$, and therefore also $f^\prime_C$, would be positive on $(0,t_0)$, resulting in a positive value for $f_C(t_0)$. This contradicts the minimality of $t_1$, hence $f_C^{\prime\prime}>0$, and $f_C,f_C^\prime>0$ then follows from the initial conditions.
		
		\smallskip	
		\noindent(2) We have $\lim_{t\to\infty}f_C(t)=\infty$ since $f_C^\prime,f_C^{\prime\prime}>0$. Now set $u=f_C^\prime$. A calculation shows that $u$ satisfies the differential equation
		\[u^{\prime\prime}+2h_0 e^{-\frac{1}{2}h_0^2}u^\prime-Ce^{-h_0^2}u=0. \]
		
		The function $h_0$ is monotone increasing, hence it has an inverse $h_0^{-1}\colon [s_0,\infty)\to[0,\infty)$, where $s_0=h_0(0)>0$. We now define $v\colon [s_0,\infty)\to[0,\infty)$ by
		\[v(s)=u(h_0^{-1}(s)). \]
		Then $v$ satisfies the differential equation
		\begin{equation}
			\label{EQ:DIFF_EQN_V}
			v^{\prime\prime}(s)+s v^\prime(s)-C v(s)=0.
		\end{equation}
		Further, the initial values for $v$ are $v(s_0)=0$ and $v^\prime(s_0)=Ce^{-\frac{1}{2}s_0^2}>0$. It follows that $v^{\prime\prime}(s_0)=-Cs_0e^{-\frac{1}{2}s_0^2}<0$.
		
		Suppose there is $s>s_0$ such that $v^\prime(s)=0$ and let $s_1$ be the smallest such $s$. Since $v^\prime(s)>0$ for $s\in[s_0,s_1)$ and $v(s_0)=0$, it follows that $v(s_1)>0$ and hence $v^{\prime\prime}(s_1)=Cv(s_1)>0$. But $v^\prime(s)>0$ for $s\in[s_0,s_1)$ and $v^\prime(s_1)=0$, so $v^{\prime\prime}(s_1)\leq0$, which is a contradiction. Hence $v^\prime>0$.
		
		By induction we have
		\[v^{(k+2)}(s)+sv^{(k+1)}(s)-(C-k)v^{(k)}(s)=0 \]
		for all $k\in\N_0$ and similarly as above we can now show inductively that $v^{(k+1)}$ does not change sign outside a compact set. Indeed, if $v^{(k+1)}(s)=0$ then $v^{(k+2)}(s)=(C-k)v^{(k)}(s)$, so $v^{(k+2)}$ does not change sign on the zeroes of $v^{(k+1)}$ outside a compact set. Hence, the function $v^{(k+1)}$ can change sign at most once outside this compact set.
		As a consequence, all $v^{(k)}$ converge to a limit. Let $L=\lim\limits_{s\to\infty}v(s)\in(0,\infty]$ and suppose $L<\infty$. Then we have
		\[\lim\limits_{s\to\infty}v^\prime(s)=\lim\limits_{s\to\infty}v^{\prime\prime}(s)=0 \]
		and by \eqref{EQ:DIFF_EQN_V} it follows that
		\[\lim\limits_{s\to\infty}sv^\prime(s)=CL. \]
		In particular, there is $s_1>s_0$ such that $sv^\prime(s)>\frac{CL}{2}$ for all $s>s_1$. Hence, for $s\geq s_1$,
		\[v(s)=v(s_1)+\int_{s_1}^{s}v^\prime(r)dr>\int_{s_1}^{s}\frac{CL}{2r}dr=\frac{CL}{2}(\ln(s)-\ln(s_1))\longrightarrow\infty \]
		as $s\to\infty$, which is a contradiction. It follows that $L=\infty$ and therefore also $\lim\limits_{s\to\infty}u(s)=\infty$.
		
		\smallskip
		\noindent(3) We define the function $y\colon [s_0,\infty)\to\R$ by
		\[y(s)=e^{-\frac{1}{2}s^2}f_C(h^{-1}_0(s))\]
		and we need to show that $y$ converges to $0$ as $s\to\infty$. A calculation shows that $y$ satisfies
		\[y^{\prime\prime}(s)=-sy^\prime(s)+(C-1)y(s) \]
		with $y(s_0)=e^{-\frac{1}{2}s_0^2}>0$ and $y^\prime(s_0)=-s_0e^{-\frac{1}{2}s_0^2}<0$. By definition of $y$ we have $y>0$ and similarly as before we conclude that $y^\prime<0$ and $y$ and all its derivatives converge. Since $y>0$ and $y^\prime<0$, the limit $L$ of $y$ is finite and non-negative. In particular, both $y^\prime$ and $y^{\prime\prime}$ converge to $0$, and by a similar argument as above it follows that $L=0$.
		
		\smallskip
		\noindent(4) We set $w=\frac{f_C^\prime}{f_C}$. Then the function $w$ satisfies
		\[w^\prime=Ce^{-h_0^2}-w^2 \]
		with $w(0)=0$. We define $z\colon [s_0,\infty)\to\R$ by
		\[ z(s)=w(h_0^{-1}(s))\frac{e^{\frac{1}{2}s^2}}{s} \]
		and we need to show that $z\in[0,1]$. A calculation shows that $z$ satisfies
		\[z^\prime(s)=-s z(s)^2+\frac{s^2-1}{s}z(s)+\frac{C}{s} \]
		with $z(s_0)=0$ and hence $z^\prime(s_0)=\frac{C}{s_0}>0$, i.e.\ $z(s)\in(0,1)$ for $s$ near $s_0$. If $z(s)=0$ for $s>s_0$, then $z^\prime(s)=\frac{C}{s}>0$ and if $z(s)=1$, then $z^\prime(s)=\frac{C-1}{s}<0$. This shows that $z$ cannot leave the interval $(0,1)$.
	\end{proof}
	\begin{definition}
		For $a,b>0$ define $h_a=a\cdot h_0$ and $f_{b,C}=b\cdot f_C$.
	\end{definition}
	The boundary conditions \eqref{EQ:BOUND_COND_0}--\eqref{EQ:BOUND_COND_t0'} can easily be satisfied with $h=h_a$ and $f=f_{b,C}$ by suitable choices of $a$ and $b$, except perhaps the value of $f$ at $t_0$, but this can be achieved for example by extending $f$ and $h$ by straight lines provided that the value of $f$ at $t_0$ is less than $\beta N\sin(R/N)$. This is the reason why the constant $\kappa$ appears in Theorem \ref{T:GEN_SURG}.
	
	Now we consider the Ricci curvatures \eqref{EQ:D_WARP_T}--\eqref{EQ:D_WARP_W}.
	\begin{lemma}
		\label{L:tb}
		For $a, b, C$ small enough we have that $\Ric(\partial_t,\partial_t), \Ric(V,V)>0$ for all $t$ and $\Ric(W,W)>0$ for small values of $t$. Further, let $t_b>0$ be the smallest value such that $\Ric(W,W)(t_b)=0$. Then
		\begin{enumerate}
			\item $t_b\to\infty$ as $b\to0$,
			\item $f_{b,C}^\prime(t_b)\to 1$ as $b\to0$.
		\end{enumerate}
	\end{lemma}
	\begin{proof}
		We calculate
		\begin{align*}
			\Ric(\partial_t,\partial_t)&=((p-1)-(q-1)C)e^{-h_0^2},
		\end{align*}
		which is positive when $C<\frac{p-1}{q-1}$. We also have
		\begin{align*}
			\Ric(V,V)&=e^{-h_0^2}+(p-2)\frac{a^{-2}-e^{-h_0^2}}{h_0^2}-(q-1)\frac{f_C^\prime h_0^\prime}{f_C h_0}.
		\end{align*}
		For $t=0$ this expression is positive if $a^{-2}>e^{-h_0(0)^2}$. Now suppose there is $t>0$ such that $\Ric(V,V)(t)=0$ and let $t_a$ be the minimal such $t$, which is positive if $a^{-2}>e^{-h_0(0)^2}$. At $t=t_a$ we then have
		\begin{align*}
			\frac{(p-2)}{a^2}&=h_0(t_a)^2\left(-e^{-h_0(t_a)^2}+(p-2)\frac{e^{-h_0(t_a)^2}}{h_0(t_a)^2} + (q-1)\frac{f_C^\prime h_0^\prime}{f_C h_0}(t_a)  \right) \\
			&=-h_0(t_a)^2e^{-h_0(t_a)^2}+(p-2)e^{-h_0(t_a)^2}+(q-1)\frac{f_C^\prime}{f_C h_0 h_0^\prime}(t_a) h_0(t_a)^2e^{-h_0^2(t_a)}.
		\end{align*}
		The left-hand side converges to $\infty$ as $a\to0$, while the right-hand side is uniformly bounded by Lemma \ref{L:FC_PROPERTIES}. Hence, by choosing $a$ sufficiently small, we can achieve that $\Ric(V,V)(t)>0$ for all $t$.
		
		Finally, we consider \eqref{EQ:D_WARP_W}:
		\begin{align*}
			\Ric(W,W)&=-Ce^{-h_0^2}+(q-2)\frac{b^{-2}-(f_C^\prime)^2}{f_C^2}-(p-1)\frac{h_0^\prime f_C^\prime}{h_0 f_C}.
		\end{align*}
		By choosing $C$ or $b$ small enough, we can achieve that $\Ric(W,W)(t)>0$ at $t=0$. At $t=t_b$ we have
		\begin{align*}
			\frac{(q-2)}{b^2}&=f_C(t_b)^2\left(Ce^{-h_0(t_b)^2}+(q-2)\frac{f_C^\prime(t_b)^2}{f_C(t_b)^2}+(p-1)\frac{f_C^\prime h_0^\prime}{f_C h_0 }(t_b) \right)\\
			&=Cf_C(t_b)^2 h_0^\prime(t_b)^2+(q-2)f_C^\prime(t_b)^2+(p-1)\frac{f_C^\prime}{f_C h_0 h_0^\prime}(t_b) f(t_b)^2 h_0^\prime(t_b)^2.
		\end{align*}
		By Lemma \ref{L:FC_PROPERTIES} both the first and third term are uniformly bounded, so $f_C^\prime(t_b)\to\infty$ as $b\to0$. Hence, $t_b\to\infty$ as $b\to 0$.
		
		Rearranging the terms yields
		\begin{align*}
			f_{b,C}^\prime(t_b)^2=b^2 f_C^\prime(t_b)&=1-\frac{b^2}{(q-2)}f_C(t_b)^2h_0^\prime(t_b)^2\left(C+(p-1)\frac{f_C^\prime}{f_C h_0 h_0^\prime}(t_b) \right)
		\end{align*}
		and the claim now follows from Lemma \ref{L:FC_PROPERTIES}.
	\end{proof}
	\begin{proof}[Proof of Theorem \ref{T:GEN_SURG}]
		By Lemma \ref{L:THM_A} it remains to show that there are values of $a,b,C$ for which $h=h_a$ and $f=f_{b,C}$ satisfy the boundary conditions \eqref{EQ:BOUND_COND_0}--\eqref{EQ:BOUND_COND_t0'} and for which the Ricci curvatures \eqref{EQ:D_WARP_T}--\eqref{EQ:D_WARP_W} are positive. By Lemma \ref{L:tb}, for $a,b,C$ sufficiently small, the Ricci curvatures are positive on $[0,t_b)$. By perhaps choosing $a$ and $b$ even smaller, we can achieve that \eqref{EQ:BOUND_COND_0} and \eqref{EQ:BOUND_COND_0'} are satisfied and such that $f^\prime(t_b)>\cos(R/N)$. Note that for \eqref{EQ:BOUND_COND_0} to be satisfied the value of $b/a$ depends on $\lambda$ and $r$. Now choose $t_1<t_b$ such that $f^\prime(t_1)>\cos(R/N)$, i.e.\ the Ricci curvatures on $[0,t_1]$ are strictly positive. We now extend the functions $f$ and $h$ as follows: The function $h$ gets extended by the constant function $h(t_1)$ and $f$ gets extended continuously such that the following holds:
		\begin{enumerate}
			\item $f^\prime_-(t_1)\geq f^\prime_+(t_1)$,
			\item $f$ is smooth on $(t_1,\infty)$,
			\item $f^\prime(t)\in(\cos(R/N),1)$ for $t\geq t_1$,
			\item $f^{\prime\prime}(t)<0$ for $t\geq t_1$.
		\end{enumerate}
		Then clearly all Ricci curvatures are positive. We now choose $t_0>t_1$ such that\linebreak $f(t_0)=N\sin(R/N)\frac{h(t_1)}{\rho}$, which exists if and only if $f(t_1)<N\sin(R/N)\frac{h(t_1)}{\rho}$. This is the case if and only if
		\[\frac{\rho}{N}<\frac{h(t_1)}{f(t_1)}\sin\left(\frac{R}{N}\right). \]
		The values of $f$ and $h$ at $t=t_1$ only depend on $a,b,C$ and $\cos(R/N)$, which in turn only depend on $p,q,\lambda,r$ and $R/N$. The value of $\lambda$ only depends on the metric $g_B$. Smoothing the functions $f$ and $h$ at $t=t_1$ using Corollary \ref{C:WARP_GLUING} finishes the proof.
	\end{proof}

	We conclude this section by noting that the warping functions constructed in \cite{SY91} and \cite{Wr98} define a doubly warped product metric with merely non-negative Ricci curvature, therefore making it necessary to apply the deformation results of Ehrlich \cite{Eh76} to obtain strictly positive Ricci curvature. In our situation the doubly warped product metric is constructed to have strictly positive Ricci curvature, thus we do not need to apply these deformation results. This is achieved in two ways: First, we introduce the factors $a$, $b$ and $C$, where decreasing one of these factors results in increasing one of the Ricci curvatures \eqref{EQ:D_WARP_T}--\eqref{EQ:D_WARP_W}. Second, by using Theorem~\ref{T:PER_GLUING} and Corollary~\ref{C:WARP_GLUING} we do not need the metrics to glue smoothly and the functions to be smooth everywhere. In particular, we do not need to make additional adjustments to the functions near the boundary points of the interval. Therefore we can achieve strictly positive Ricci curvature by simply choosing the factors $a$, $b$ and $C$ small enough.
	
	\section{Proof of Theorems \ref{T:PLUMBING} and \ref{T:2-SPHERE_BUNDLES}}
	\label{S:THMS_BC}
	
	First we prove Theorem \ref{T:PLUMBING}. We denote the sphere bundle of the disc bundle $\overline{E}_i$ by $E_i$ and we set $q=\dim(B_1)$ and $p$ as the fiber dimension of $\overline{E}_1$. Since all $B_i$, $i\geq2$, have core metrics, the manifold $\partial W$ is obtained by iterated surgeries on the manifold $M=E_1$ as in Theorem \ref{T:GEN_SURG} (cf.\ \eqref{EQ:PLUMB_BOUND}). By a deformation result of Gao and Yau \cite{GY86} for negative Ricci curvature, that can easily be transferred to positive Ricci curvature (see also \cite[Theorem 1.10]{Wr02}) for every $x\in B_1$ and any open neighborhood $U$ of $x$ we can deform the metric on $B_1$ to agree with the original metric on $B_1\setminus U$ and to have constant sectional curvature 1 on a neighborhood of $x$. Hence, for any $k_1\in\N$ we can deform the metric on $B_1$ such that there are positive constants $R_1,\dots,R_{k_1}$ and an isometric embedding
	\[D^{q}_{R_1}(1)\sqcup\dots\sqcup D^{q}_{R_{k_1}}(1)\hookrightarrow B_1. \]
	Now we equip $E_1$ with the metric $g_{E_1}(\rho,\theta)$ according to a connection $\theta$ that is flat over each embedded disc, so we have an isometric embedding
	\[S^{p-1}(\rho)\times D^{q}_{R_1}(1)\sqcup\dots\sqcup S^{p-1}(\rho)\times D^{q}_{R_{k_1}}(1)\hookrightarrow E_1. \]
	By Proposition \ref{P:VILMS_RICCI} there is a constant $\rho_1>0$ such that $g_{E_1}(\rho,\theta)$ has positive Ricci curvature for all $\rho<\rho_1$. By choosing $\rho$ small enough we can satisfy the assumptions of Theorem \ref{T:GEN_SURG}. By possibly choosing $\rho$ even smaller we can freely choose the radii of the fibers of the bundles we attach. Hence, by choosing sufficiently small radii for the attached bundles, we can satisfy again the assumptions of Theorem \ref{T:GEN_SURG} for the attached bundles.
	
	We repeat this process: Since we glue according to a tree, where we consider $E_1$ as the root, the manifold $\partial W$ is obtained by successively gluing the bundles that have distance $i$ from the root to the bundles that have distance $i-1$ from the root. As above we can apply Theorem~\ref{T:GEN_SURG} for each gluing by possibly decreasing the fiber radii of all the preceding bundles. This finishes the proof of the first part of Theorem \ref{T:PLUMBING}.
	
	To show the second part we use the following.
	\begin{proposition}[{\cite[Theorem 2.5]{Bu19}}]
		\label{P:DISC_SPHERE_GLUE}
		For $q\geq3$, $p\geq4$, $R>1$, and any $\nu>0$ sufficiently small, there is a core metric on $D^{q}\times S^{p-1}$ such that the boundary is isometric to $R^2ds^2_{q-1}+ds^2_{p-1}$ and the principal curvatures of the boundary are all at least $-\nu$.
	\end{proposition}
	If $B_1$ admits a core metric, then we can choose the embeddings on which we perform the surgeries to be disjoint from the embedded disc $\varphi(D^{q})$. We can also assume that the connection for the bundle $E_1$ is flat over $\varphi(D^{q})$. Hence, if we remove $\varphi(D^{q})$ from $B_1$ and the corresponding part of the bundle $E_1$, we obtain a boundary component isometric to $ds_{q-1}^2+\rho ds_{p-1}^2$ with non-negative definite second fundamental form. By \cite[Proposition 2.3]{Bu19} we can assume that the second fundamental form is positive definite and by possibly choosing $\rho$ smaller and rescaling we can assume that the boundary is isometric to $R^2ds_{q-1}^2+ds_{p-1}^2$ for some $R>1$. Hence, by Theorem \ref{T:PER_GLUING}, we can glue with the metric from Proposition \ref{P:DISC_SPHERE_GLUE} (where we assume $p\geq4$) and obtain a core metric on $\partial W$. This finishes the proof of Theorem \ref{T:PLUMBING}.
	\hfill$\square$
	
	\leavevmode\\
	
	For the proof of Theorem \ref{T:2-SPHERE_BUNDLES} let $E\to B^q$ be a linear $S^p$-bundle, where $B$ is a closed manifold that admits a core metric. First suppose that $p=2$ and $q\geq4$. Let $\overline{E}$ be the disc bundle of $E$. Let $\overline{M}$ be the manifold obtained by plumbing as follows:
	\[
	\begin{tikzcd}[cells={nodes={ellipse, draw=black, anchor=center,minimum height=2em}}]
		\overline{E} \arrow[dash]{r}{+} & S^3\times D^{q}\arrow[dash]{r}{+}& S^{q}\times D^3
	\end{tikzcd}
	\]
	According to \eqref{EQ:PLUMB_BOUND} we have
	\begin{align*}
		\partial\overline{M}&\cong E\setminus\pi^{-1}(D^q)^\circ\cup_{S^{q-1}\times S^2}(S^{q-1}\times S^3\setminus(D^3\sqcup D^3)^\circ)\cup_{S^{q-1}\times S^2} (S^q\setminus (D^q)^\circ)\times S^2\\
		&\cong E\setminus\pi^{-1}(D^q)^\circ\cup_{S^{q-1}\times S^2}(S^{q-1}\times [0,1]\times S^2)\cup_{S^{q-1}\times S^2}D^q\times S^2\\
		&\cong E\setminus\pi^{-1}(D^q)^\circ\cup_{S^{q-1}\times S^2}D^q\times S^2\\
		&\cong E,
	\end{align*}
	see also \cite[Lemma 2.10]{CW17}. By applying Theorem \ref{T:PLUMBING} with $\overline{E}_1=S^3\times D^q$, we obtain a core metric on $\partial\overline{M}$ and thus on $E$.
	
	Now suppose that $q=2$ and $p\geq4$. Then $B$ is a 2-dimensional closed manifold with a core metric, hence $B\cong S^2$. Since $\pi_1(\mathrm{SO}(p+1))\cong\Z/2$, there are precisely two isomorphism classes of linear $S^p$-bundles over $S^2$. If $E$ is the trivial bundle, i.e.\ $E\cong S^2\times S^p$, then we can also consider it as a linear $S^2$-bundle over $S^p$ and apply the first part of Theorem \ref{T:2-SPHERE_BUNDLES}. If $E$ is the non-trivial bundle then the claim follows from Theorem \ref{T:PLUMBING} and Lemma \ref{L:PLUMBING} below.
	\begin{lemma}
		\label{L:PLUMBING}
		Let $\overline{M}$ be the manifold obtained by plumbing as follows.
		\[
		\begin{tikzcd}[cells={nodes={ellipse, draw=black, anchor=center,minimum height=2em}}]
			\C P^2\times D^{p-1} \arrow[dash]{r}{+} & S^{p-1}\times D^{4}
		\end{tikzcd}
		\]
		Then $\partial \overline{M}$ is diffeomorphic to the unique non-trivial linear $S^{p}$-bundle over $S^2$.
	\end{lemma}
	\begin{proof}
		According to \eqref{EQ:PLUMB_BOUND} we have
		\[\partial\overline{M}\cong (\C P^2\setminus (D^4)^\circ)\times S^{p-2}\cup_{S^3\times S^{p-2}}S^3\times D^{p-1}. \]
		The manifold $\C P^2\setminus (D^4)^\circ$ is diffeomorphic to the disc bundle of the tautological line bundle over $\C P^1\cong S^2$. Hence, the manifold $(\C P^2\setminus (D^4)^\circ)\times S^{p-2}$ has the structure of a fiber bundle over $S^2$ with fiber $D^2\times S^{p-2}$. On the other hand, the manifold $S^3\times D^{p-1}$ also has the structure of a fiber bundle over $S^2$ obtained by the Hopf fibration $S^3\to S^2$, i.e.\ the fiber of this bundle is $S^1\times D^{p-1}$. Since the bundle projection $\partial (\C P^2\setminus (D^4)^\circ)\cong S^3\to S^2$ is also given by the Hopf fibration, we glue fibers to fibers, so $\partial \overline{M}$ has the structure of a fiber bundle over $S^2$ with fiber
		\[D^2\times S^{p-2}\cup_{S^1\times S^{p-2}} S^1\times D^{p-1}\cong S^p. \]
		Both bundles have the structure group of the Hopf fibration, which is contained in $\mathrm{SO}(2)$, hence $\partial\overline{M}$ also has structure group contained in $\mathrm{SO}(2)\subseteq \mathrm{SO}(p+1)$, so it is a linear bundle. It is non-trivial, since under the inclusion $(\C P^2\setminus (D^4)^\circ)\times S^{p-2}\hookrightarrow\partial\overline{M}$, the class $w_2(\partial\overline{M})$ gets mapped to $w_2((\C P^2\setminus (D^4)^\circ)\times S^{p-2})$, which is\linebreak non-trivial as it is the pullback of $w_2(\C P^2\times S^{p-2})$ under the inclusion\linebreak $(\C P^2\setminus (D^4)^\circ)\times S^{p-2} \hookrightarrow \C P^2\times S^{p-2}$ (which is an isomorphism on $H^2$). Thus $w_2(\partial\overline{M})$ is non-trivial, so $\partial\overline{M}$ cannot be diffeomorphic to $S^2\times S^{p}$.
	\end{proof}

	\section{Applications in Dimension 6}
	\label{S:6-MANIFOLDS}
	
	In this section we consider applications in dimension 6. Corollary \ref{C:CORE_METRICS_DIM6} directly follows from Theorem \ref{T:CORE_METRICS} and Theorem \ref{T:2-SPHERE_BUNDLES}. We will now show that we can construct new examples of manifolds with a metric of positive Ricci curvature in this way.
	
	\subsection{Known Examples of Closed, Simply-Connected 6-Manifolds with a Metric of Positive Ricci Curvature}
	\label{SS:KNOWN}
	
	Let us consider the examples known so far. By \cite[Theorem 3.5]{Na79} (cf.\ also Proposition \ref{P:VILMS_RICCI}), fiber bundles with homogeneous fibers admit a metric of positive Ricci curvature if both base and fiber admit a metric of positive Ricci curvature. We obtain the following list of manifolds that admit a metric of positive Ricci curvature:
	\begin{enumerate}
		\item Linear $S^2$-bundles over $B=\#_k (\pm\C P^2)$ (if $k=0$ then $B=S^4$) or $B=\#_k(S^2\times S^2)$ (the base $B$ admits a metric of positive Ricci curvature by \cite{Pe97} and by \cite[Theorem 4]{SY91});
		\item $S^3\times S^3$;
		\item $S^2\ttimes S^4$, the unique non-trivial linear $S^4$-bundle over $S^2$;
		\item Projective bundles, i.e.\ $\C P^2$-bundles, over $S^2$.
	\end{enumerate}
	
	Next, we list all closed simply-connected 6-dimensional homogeneous spaces, cohomogeneity one manifolds and biquotients that are not already contained in the list above, by using the classification results of Gorbatsevitch \cite{Go80}, Hoelscher \cite{Ho10}, and DeVito \cite{DV17}, ordered by their second Betti number. These manifolds are:
	\begin{enumerate}
		\item[(5)] $S^6$;
		\item[(6)] The oriented Grassmannian $\tilde{G}_2(\R^5)\cong \mathrm{SO}(5)/(\mathrm{SO}(3)\times \mathrm{SO}(2))$ (which is a homogeneous space);
		\item[(7)] The homogeneous space $\mathrm{SU}(3)/T^2$ and the biquotient $\mathrm{SU}(3)\bq T^2$;
		\item[(8)] Biquotients of the form $(S^3)^3\bq T^3$ that are diffeomorphic to a $(\C P^2\#\C P^2)$-bundle over $S^2$ or to one of 4 sporadic examples (see \cite[Proposition 4.23]{DV17}, and note that the other families appearing in this proposition are already contained in the previous items).
	\end{enumerate}

	Further, as a result of Yau's proof of the Calabi conjecture, Fano varieties admit metrics of positive Ricci curvature. In (complex) dimension 3 Fano varieties were classified by Iskovskih \cite{Is77,Is78} for $b_2=1$ and by Mori and Mukai \cite{MM81,MM03} for $b_2\geq 2$. Their result can be summarized as follows (note that we omit the manifolds with $b_2>5$ as they are all diffeomorphic to a product of $S^2$ and and a connected sum of copies of $\pm \C P^2$, i.e.\ they are contained in item (1)): 
	
	\begin{enumerate}
		\item[(9)] 18 types of Fano 3-folds with $b_2=1$ and 83 types of Fano 3-folds with $2\leq b_2\leq5$.
	\end{enumerate}
	
	Finally, Sha and Yang \cite{SY91}, by using a surgery theorem similar to Theorem \ref{T:WR-SURGERY}, and Corro and Galaz-Garc\'ia \cite{CG20}, by using a lifting result of Gilkey, Park and Tuschmann \cite{GPT98}, obtained metrics of positive Ricci curvature on connected sums of sphere bundles, which in dimension 6 are given as follows:
	\begin{enumerate}
		\item[(10)] $\#_k (S^2\times S^4) \#_l (S^3\times S^3)$,
		\item[(11)] $(S^2\ttimes S^4)\#_{k}(S^2\times S^4)\#_{k+2} (S^3\times S^3)$.
	\end{enumerate}
	
	Note that the manifolds in items (9), (10) and (11), apart from $S^3\times S^3$, are the only manifolds in this list with non-trivial third Betti number, and, together with the linear $S^2$-bundles in item (1) they are the only manifolds where the second Betti number is greater than 3.
	
	\subsection{New Examples of Closed, Simply-Connected 6-Manifolds with a Metric of Positive Ricci Curvature}
	
	Given a closed, simply-connected and oriented 6-manifold $M$ we have an associated trilinear form $\mu_M\colon H^2(M)\otimes H^2(M)\otimes H^2(M)\to\Z$ defined by
	\[\mu_M(x,y,z)=\langle x\smile y\smile z,[M]\rangle.\]
	Further invariants are the third Betti number $b_3(M)\in 2\Z$, the second Stiefel-Whitney class $w_2(M)\in H^2(M;\Z/2)\cong H^2(M)\otimes\Z/2$ and the first Pontryagin class $p_1(M)\in H^4(M)\cong \mathrm{Hom}(H^2(M),\Z)$. In fact, if $H_2(M)$ is torsion-free, the invariants $(H^2(M),b_3(M),\mu_M,w_2(M),p_1(M))$ already determine the diffeomorphism type of $M$ by the classification of Jupp \cite{Ju73}.
	
	The manifolds we consider are linear $S^2$-bundles over a closed, simply-connected 4-manifold $B$. We have the following classification result.
	\begin{lemma}
		Let $B$ be a closed, simply-connected 4-manifold. Then isomorphism classes of linear $S^2$-bundles over $B$ are in bijection with pairs $(x,Y)\in H^2(B;\Z/2)\times H^4(B)$ such that $X^2\equiv Y\mod 4$ for some $X\in H^2(B)$ with $X\equiv x\mod 2$. The bijection is given by assigning the pair $(w_2(\xi),p_1(\xi))$ to the sphere bundle of a vector bundle $\xi$ over $B$ of rank 3.
	\end{lemma}
	\begin{proof}
		This is precisely the classification of Dold and Whitney \cite{DW59}, except that the condition on $(x,Y)$ is given by
		\[\mathcal{P}_2(x)\equiv Y\mod 4, \]
		where $\mathcal{P}_2\colon H^2(B;\Z/2)\to H^4(B;\Z/4)$ is the Pontryagin square operation. Since $H^2(B)$ is free abelian, every class $x\in H^2(B;\Z/2)$ has a preimage $X\in H^2(B)$. Hence,
		\[\mathcal{P}_2(x)=\mathcal{P}_2(X\textrm{ mod } 2)=X^2\mod 4, \]
		which shows that the condition $\mathcal{P}_2(x)\equiv Y\mod 4$ is equivalent to $X^2\equiv Y\mod 4$ for one, and thus for all preimages $X\in H^2(B)$ of $x$.
	\end{proof}
	
	Now we consider the special case $B=B_{\overline{\gamma}}=\#_{i=1}^k \gamma_i \C P^2$ for $\overline{\gamma}=(\gamma_1,\dots,\gamma_k)\in\{\pm1\}^k$. We denote by $b_i\in H^2(\gamma_i\C P^2)$ a generator of the cohomology ring of the $i$-th summand of $B$ and let $c=[B]^*\in H^4(B)$. Then $b_i\smile b_j=\delta_{ij}\gamma_ic$ and we have $w_2(B)\equiv\sum_i b_i\mod 2$ and $p_1(B)=\sum_i3\gamma_ic$.
	\begin{corollary}
		\label{C:2-SPHERE_BUNDLES}
		Isomorphism classes of linear $S^2$-bundles over $B_{\overline{\gamma}}$ are in bijection with elements $(\alpha,\overline{\beta})\in\Z\times\{0,1\}^k$. The bijection assigns to $\alpha$ and $\overline{\beta}=(\beta_1,\dots,\beta_k)$ the sphere bundle of a vector bundle $\xi$ over $B_{\overline{\gamma}}$ of rank 3 with first Pontryagin class $p_1(\xi)=(4\alpha+\sum_i\gamma_i\beta_i)c$ and second Stiefel-Whitney class $w_2(\xi)=\sum_i\beta_i b_i\mod 2$.
	\end{corollary}
	\begin{definition}
		We denote the total space of the $S^2$-bundle we obtain as in Corollary \ref{C:2-SPHERE_BUNDLES} by $M_{\alpha,\overline{\beta},\overline{\gamma}}$.
	\end{definition}
	By Theorem \ref{T:2-SPHERE_BUNDLES}, all the manifolds $M_{\alpha,\overline{\beta},\overline{\gamma}}$ admit core metrics.\pagebreak
	
	\begin{proposition}
		\label{P:NEW}
		Let $M=\#_{i=1}^m (S^3\times S^3)\#_{i=1}^l M_{\alpha_i,\overline{\beta}_i,\overline{\gamma}_i}$. Suppose that
		\begin{itemize}
			\item $m,l\geq1$, and $\alpha_1\neq0$ or $k_1\geq1$, or
			\item $l\geq2$ and $k_1,k_2\geq1$.
		\end{itemize}
		Then $M$ is not diffeomorphic to the total space of a linear sphere bundle or a projective bundle as in item (1) of Subsection \ref{SS:KNOWN}, a homogeneous space, a cohomogeneity one manifold or a biquotient. In particular it is an example of a manifold with a metric of positive Ricci curvature that is not contained in the list of Section \ref{SS:KNOWN} except for finitely many possibilities that are contained in item (9).
	\end{proposition}

	Note that the construction of the manifold $M$ in Proposition \ref{P:NEW} differs from that in item (1) as we take connected sums of sphere bundles, whereas the manifolds in item (1) are sphere bundles over manifolds obtained by connected sums. 

	\begin{remark}
		\label{R:NEW}
		It follows from Corollary \ref{C:M_ALPHA_BETA_GAMMA} below that $M_{\alpha,\overline{\beta},\overline{\gamma}}$ is spin if and only if $\beta_{j}=1$ for all $j$. In Proposition \ref{P:NEW} there is no restriction on the $\overline{\beta}_i$, hence we obtain infinitely many new examples both in the spin and non-spin case.
	\end{remark}

	To prove Proposition \ref{P:NEW} we need to determine the invariants of the manifolds $M_{\alpha,\overline{\beta},\overline{\gamma}}$, which will be carried out in the next section.
	
	\subsection{The Invariants of Linear $S^2$-Bundles over Closed, Simply-Connected 4-Manifolds}
	\label{SS:S2-BUNDLES}
	
	We will now determine the invariants of the total space of a linear $S^2$-bundle over a closed, simply-connected 4-manifold in order to prove Proposition~\ref{P:NEW}.
	
	Recall that for a linear $S^m$-bundle $\pi\colon E\to B$ we have the Gysin sequence
	\[\dots\xrightarrow{\cdot\smile e(\pi)} H^i(B)\xrightarrow{\pi^*}H^i(E)\xrightarrow{\psi}H^{i-m}(B)\xrightarrow{\cdot\smile e(\pi)}H^{i+1}(B)\xrightarrow{\pi^*}\dots, \]
	where $e(\pi)\in H^{m+1}(B)$ is the Euler class. The map $\psi\colon H^*(E)\to H^{*-m}(B)$ satisfies the following property.
	\begin{lemma}[{\cite[Lemma 1]{Ma58}}]
		\label{L:GYSIN_FIB_INT_MAP}
		For $x\in H^i(B)$ and $y\in H^j(E)$ we have
		\[\psi(\pi^*(x)\smile y)=(-1)^i x\smile \psi(y). \]
	\end{lemma}
	Now assume that the Euler class $e(\pi)$ vanishes. Then the Gysin sequence splits up into short exact sequences of the form
	\begin{equation}
		\label{EQ:GYSIN_SPLIT}
		0\longrightarrow H^i(B)\xrightarrow{\pi^*}H^i(E)\xrightarrow{\psi} H^{i-m}(B)\longrightarrow0.
	\end{equation}
	From now on we assume that $B$ is connected. Then, following \cite[Section 8]{Ma58}, let $a\in H^m(E)$ such that $\psi(a)=1\in H^0(B)\cong \Z$. We define $\theta_a\colon H^*(B)\to H^{*+m}(E)$ by
	\[\theta_a(x)=(-1)^{i(m+1)}a\smile\pi^*(x) \]
	for $x\in H^i(B)$. Then, by Lemma \ref{L:GYSIN_FIB_INT_MAP}, the map $\theta_a$ defines a splitting of \eqref{EQ:GYSIN_SPLIT}, i.e.\ $\psi\circ\theta_a=\textrm{id}_{H^*(B)}$. In particular, we have
	\begin{equation}
		\label{EQ:COHOM_SPLITTING}
		H^i(E)=\pi^*(H^i(B))\oplus \theta_a(H^{i-m}(B)).
	\end{equation}
	
	Now we assume $m=2$, i.e.\ $\pi\colon E\to B$ is a linear $S^2$-bundle, and that $B$ is a closed simply-connected 4-manifold, so $E$ is a simply-connected closed 6-manifold. We choose an orientation on $B$ and orient $E$ such that $\psi([E]^*)=[B]^*$. We have that $H^3(B)\cong H_1(B)$ is trivial, hence $e(\pi)\in H^3(B)$ vanishes. We denote by $\xi$ the vector bundle corresponding to $\pi$.
	\begin{lemma}
		\label{L:COHOM_2_SPHERE_BDL}
		The manifold $E$ has torsion-free homology.	Let $W\in H^2(B)$ such that $W\equiv w_2(\xi)\mod 2$. Then there exists $a\in H^2(E)$ with $\psi(a)=1$ such that
		\[H^2(E)=\pi^*(H^2(B))\oplus \Z a \]
		and for $x_1,x_2,x_3\in H^2(B)$ we have
		\begin{alignat*}{2}
			&\mu_E(\pi^*x_1,\pi^*x_2,\pi^*x_3) &&= 0,\\
			&\mu_E(\pi^*x_1,\pi^*x_2,a) &&= \langle x_1\smile x_2,[B]\rangle,\\
			&\mu_E(\pi^*x_1,a,a) &&= \langle x_1\smile W,[B]\rangle,\\
			&\mu_E(a,a,a) &&= \frac{1}{4}\langle 3W^2+p_1(\xi),[B]\rangle.
		\end{alignat*}
		Further, we have $b_3(E)=0$ and for $x\in H^2(B)$ we have
		\begin{alignat*}{2}
			&w_2(E)&&=\pi^*w_2(\xi)+\pi^*w_2(B),\\
			&p_1(E)(\pi^*x)&&=0,\\
			&p_1(E)(a)&&=\langle p_1(\xi)+p_1(B),[B]\rangle.
		\end{alignat*}	
	\end{lemma}
	\begin{proof}
		Since $B$ is simply-connected, the group $H^2(B)$ is torsion-free. Hence, by \eqref{EQ:COHOM_SPLITTING}, the manifold $E$ has torsion-free cohomology, and thus, by Poincar\'e duality, also torsion-free homology. The splitting $H^2(E)=\pi^*(H^2(B))\oplus \Z a$ also follows from \eqref{EQ:COHOM_SPLITTING} and holds for any $a\in H^2(E)$ with $\psi(a)=1$.
		
		For $x_1,x_2,x_3\in H^2(B)$ we have
		\[\pi^*x_1\smile \pi^* x_2\smile \pi^* x_3=\pi^*(x_1\smile x_2\smile x_3)=0. \]
		For the remaining cup products first note that
		\begin{align*}
			[B]^*\frown\pi_*(a\frown [E])=\pi_*((\pi^*([B]^*)\smile a)\frown [E])&=\pi_* (\theta_a([B]^*)\frown [E])\\
			&=\pi_*([E]^*\frown[E])=1,
		\end{align*}
		hence $\pi_*(a\frown[E])=[B]$. It follows that, for any $y\in H^4(B)$, we have
		\begin{align*}
			\pi_*((\pi^*y\smile a)\frown[E])=\pi_*(\pi^*(y)\frown(a\frown[E]))=y\frown\pi_*(a\frown[E])=y\frown[B].
		\end{align*}
		In particular,
		\begin{align*}
			\langle \pi^*x_1\smile\pi^*x_2\smile a,[E]\rangle=\langle x_1\smile x_2,[B]\rangle.
		\end{align*}
		For the remaining cup products we need to determine $a^2$. By \eqref{EQ:COHOM_SPLITTING} there are unique $\alpha\in H^4(B)$, $\beta\in H^2(B)$ such that
		\[a^2=\pi^*\alpha+a\smile\pi^*\beta. \]
		By \cite[(8.2) and Theorem III]{Ma58}, we can choose $a$ so that $\beta=W$ and by \cite[Theorem IV]{Ma58} we have
		\[p_1(\xi)=4\alpha+\beta^2=4\alpha+W^2. \]
		It follows that
		\[a^2=\frac{1}{4}\pi^*(p_1(\xi)-W^2)+a\smile \pi^*W, \]
		so
		\[\pi^*x_1\smile a\smile a=a\smile\pi^*(x_1\smile W) \]
		and hence
		\[\mu_E(\pi^*x_1,a,a)=\langle x_1\smile W,[B]\rangle. \]
		Further,
		\begin{align*}
			a^3&=\frac{1}{4}a\smile\pi^*(p_1(\xi)-W^2)+a^2\smile \pi^*W\\
			&=\frac{1}{4}a\smile\pi^*(p_1(\xi)-W^2)+a\smile \pi^*(W^2)\\
			&=\frac{1}{4}a\smile\pi^*(3W^2+p_1(\xi)),
		\end{align*}
		so
		\[\mu_E(a,a,a)=\frac{1}{4}\langle 3W^2+p_1(\xi),[B]\rangle. \]
		
		To show the remaining claims, denote by $\overline{E}$ the total space of the disc bundle corresponding to $E$. Then $T\overline{E}\cong TE\oplus\underline{\R}_E$, the trivial factor corresponds to the normal vector of $E=\partial\overline{E}$. Further, $T\overline{E}\cong \pi^*(TB)\oplus\pi^*\xi$. It follows that
		\[w_2(E)=\pi^*w_2(\xi)+\pi^*w_2(B) \]
		and, since $H^4(E)$ is torsion-free,
		\[p_1(E)=\pi^*p_1(\xi)+\pi^*p_1(B), \]
		from which the claims on $p_1(E)$ directly follow.
	\end{proof}

	We denote the pullback along the bundle projection of $b_i$ in $H^*(M_{\alpha,\overline{\beta},\overline{\gamma}})$ again by $b_i$. As a consequence of Lemma \ref{L:COHOM_2_SPHERE_BDL} with $W=\sum_i\beta_i b_i$ we obtain the corollary below. 
	\begin{corollary}
		\label{C:M_ALPHA_BETA_GAMMA}
		The manifold $M=M_{\alpha,\overline{\beta},\overline{\gamma}}$ is a simply-connected 6-manifold with torsion-free homology and $b_3(M)=0$. Further, we have
		\begin{alignat*}{2}
			&\mu_M(b_i,b_j,b_m) &&= 0\\
			&\mu_M(b_i,b_j,a) &&= \delta_{ij}\gamma_i,\\
			&\mu_M(b_i,a,a) &&= \beta_i\gamma_i,\\
			&\mu_M(a,a,a) &&=\alpha+\sum_i\beta_i\gamma_i,
		\end{alignat*}
		and
		\begin{alignat*}{2}
			&w_2(M)&&=\sum_i(1-\beta_i)b_i\mod 2,\\
			&p_1(M)(b_i)&&=0,\\
			&p_1(M)(a)&&=4\alpha+\sum_i(3+\beta_i)\gamma_i.
		\end{alignat*}
	\end{corollary}

	We are now ready to prove Proposition \ref{P:NEW}.
	
	\begin{proof}[Proof of Proposition \ref{P:NEW}]
		First suppose that $m,l\geq1$. Besides $S^3\times S^3$, the only manifolds in the list of Section \ref{SS:KNOWN} with non-trivial third Betti number are those in items (2), (9), (10) and (11). The manifolds in items (2), (10) and (11) have trivial trilinear form $\mu$. If $\alpha_1\neq0$ or $k_1\geq1$, then, by Corollary \ref{C:M_ALPHA_BETA_GAMMA}, the manifold $M$ has non-trivial trilinear form $\mu_M$, so it can only be diffeomorphic to the manifolds in item (9).

		Now suppose $l\geq2$ and $k_1,k_2\geq1$. Then $b_2(M)\geq4$, so $M$ can only be diffeomorphic to the manifolds in items (1), (9), (10) and (11). Every manifold $N$ in item (1) with $p_1(N)\neq0$ has the property that its trilinear form is trivial on the subspace $U=\{x\in H^2(N)\mid x\smile p_1(N)=0\}$ by Lemma \ref{L:COHOM_2_SPHERE_BDL}. In particular, after taking tensor product with $\Q$ and extending the trilinear form $\Q$-linearly, the trilinear form is trivial on a subspace of codimension 1. Since the manifolds in items (10) and (11) have trivial trilinear form, they also have this property. We conclude the proof by showing that $M$ does not have this property.
		
		Recall that $M=\#_{i=1}^m (S^3\times S^3)\#_{i=1}^l M_{\alpha_i,\overline{\beta}_i,\overline{\gamma}_i}$ and that for each $i$ we have a class $a_i\in H^2(M_{\alpha_i,\overline{\beta}_i,\overline{\gamma}_i})$ as in Lemma \ref{L:COHOM_2_SPHERE_BDL}. Let $U\subseteq H^2(M)\otimes \Q$ be a subspace of codimension $1$. The subspace $W\subseteq H^2(M)\otimes \Q$ generated by $\{a_1,b_{1,1},a_2,b_{2,1}\}$ has dimension 4, hence its intersection with $U$ has dimension 3 or 4. In the latter case we have in particular that $a_1,b_{1,1}\in U$ and $\mu_M(b_{1,1},b_{1,1},a_1)=\gamma_{1,1}\neq0$. In the first case, where $U\cap W$ has dimension 3, we can express a basis for $U\cap W$ via Gauss elimination as linear combinations of the elements $\{a_1,b_{1,1},a_2,b_{2,1}\}$ by multiplying one of the following matrices with $(a_1,b_{1,1},a_2,b_{2,1})^\intercal$:
		\[\begin{pmatrix}
			1 & 0 & 0 & \lambda_1\\
			0 & 1 & 0 & \lambda_2\\
			0 & 0 & 1 & \lambda_3
			\end{pmatrix},\quad \begin{pmatrix}
			1 & 0 & \lambda_1 & 0\\
			0 & 1 & \lambda_2 & 0\\
			0 & 0 & 0 & 1
			\end{pmatrix},\quad \begin{pmatrix}
			1 & \lambda_1 & 0 & 0\\
			0 & 0 & 1 & 0\\
			0 & 0 & 0 & 1
			\end{pmatrix},\quad \begin{pmatrix}
			0 & 1 & 0 & 0\\
			0 & 0 & 1 & 0\\
			0 & 0 & 0 & 1
		\end{pmatrix}, 
		\]
		where $\lambda_1,\lambda_2,\lambda_3\in\Q$. Using Corollary \ref{C:M_ALPHA_BETA_GAMMA} one now easily sees that in each case the trilinear form $\mu_M$ is non-trivial on $U$.
	\end{proof}
	
	\bibliographystyle{plainurl}
	\bibliography{References}

\end{document}